\documentclass[11pt]{amsart}
\usepackage{amssymb,latexsym,amsmath}
\usepackage[mathscr]{eucal}

\numberwithin{equation}{section}

\newtheorem{corollary}[equation]{Corollary}
\newtheorem*{corollary*}{Corollary}
\newtheorem{lemma}[equation]{Lemma}
\newtheorem*{lemma*}{Lemma}
\newtheorem{proposition}[equation]{Proposition}
\newtheorem*{proposition*}{Proposition}
\newtheorem{theorem}[equation]{Theorem}
\newtheorem*{theorem*}{Theorem}

\theoremstyle{remark}
\newtheorem*{assume*}{Assume}

\newtheorem*{definition*}{Definition}

\newtheorem*{example*}{Example}
\newtheorem*{hint*}{Hint}
\newtheorem*{notation*}{Notation}

\newtheorem*{remark*}{Remark}

\numberwithin{HWeq}{section}
\theoremstyle{definition}

\def\a{{\alpha}}

\def\e{\varepsilon}

\def\w{\omega}

\def\bC{\mathbb C}

\def\bO{\mathbb O}
\def\bP{\mathbb P}

\def\bR{\mathbb R}
\def\bS{\mathbb S}

\def\cC{\mathcal C}

\def\cE{\mathcal E}
\def\cF{\mathcal F}

\def\fg{{\mathfrak{g}}}

\def\fh{\mathfrak{h}}

\def\fo{\mathfrak{o}}

\def\fsu{\mathfrak{su}}

\def\tAd{\mathrm{Ad}}

\def\tAnn{\mathrm{Ann}}

\def\tCl{\mathrm{Cl}}

\def\tcodim{\mathrm{codim}}

\def\td{\mathrm{d}}

\def\tEnd{\mathrm{End}}

\def\tGr{\mathrm{Gr}}

\def\ti{\mathrm{i}}

\def\tId{\mathrm{Id}}

\def\tIm{\mathrm{Im}}

\def\tO{\mathrm{O}}

\def\sfP{\mathsf{P}}

\def\tRe{\mathrm{Re}}

\def\tSO{\mathrm{SO}}

\def\tspan{\mathrm{span}}

\def\tSpin{\mathit{Spin}}

\def\tSU{\mathrm{SU}}

\def\tvol{\mathit{vol}}
\def\sfv{\mathsf{v}}

\def\sI{\mathscr{I}}
\def\sP{\mathscr{P}}

\def\sV{\mathscr{V}}

\def\op{\oplus}
\def\ot{\otimes}

\def\dz{\mathrm{d}z}
\def\dbz{\mathrm{d} \bar{z}}
\def\lefthook{\hbox{\small{$\lrcorner\, $}}}
\def\tw{\hbox{\small $\bigwedge$}}

\def\half{\tfrac{1}{2}}
\def\ihalf{\tfrac{\ti}{2}}

\newcounter{cnt}

\begin{document}
\title[Parallel calibrations]{Parallel calibrations and minimal submanifolds}
\author{C. Robles}
\address{Department of Mathematics\\ Mail-stop 3368\\ Texas A\&M University\\ College Station, TX  77843-3368}
\email{robles@math.tamu.edu}
\thanks{I thank the NSF for partial support through DMS - 0805782.}
\date{\today}
\maketitle
\begin{abstract}  Given a parallel calibration $\varphi \in \Omega^p(M)$ on a Riemannian manifold $M$, I prove that the $\varphi$--critical submanifolds with nonzero critical value are minimal submanifolds.  I also show that the $\varphi$--critical submanifolds are precisely the integral manifolds of a $\mathscr{C}^\infty(M)$--linear subspace $\sP \subset \Omega^p(M)$.  In particular, the calibrated submanifolds are necessarily integral submanifolds of the system.  (Examples of parallel calibrations include the special Lagrangian calibration on Calabi-Yau manifolds, (co)associative calibrations on $G_2$--manifolds, and the Cayley calibration on $\tSpin(7)$--manifolds.)  
\end{abstract}

\section{Introduction}\label{S:intro}

\subsection{Calibrated geometry}\label{S:cg}

Let's begin by setting notation and reviewing (briefly) calibrated geometry.  See \cite{HL} for a through introduction.

Let $V$ be a real, $n$-dimensional vector space equipped with an inner product.  Throughout $\{e_1,\ldots,e_n\} \subset V$ will denote a set of orthonormal vectors.  Let
\begin{displaymath}
  \tGr_o(p,V) \  := \ \left\{ e_1 \wedge \cdots \wedge e_p \right\} 
  \ \subset \ \tw^p V
\end{displaymath}
denote the unit decomposable (or simple) $p$-vectors.  Notice that $\tGr_o(p,V)$ is a double cover of the Grassmannian $\tGr(p,V)$ of $p$-planes in $V$.  Given $\xi \in \tGr_o(p,V)$, let $[\xi] \in \tGr(p,V)$ denote the corresponding $p$-plane.  I will abuse terminology by referring to elements of both $\tGr_o(p,V)$ and $\tGr(p,V)$ as $p$-planes.  (Properly, elements of $\tGr_o(p,V)$ are oriented $p$-planes.) 

Let $M$ be an $n$-dimensional Riemannian manifold.  Let $\tGr(p,TM)$ denote the Grassmann bundle of tangent $p$-planes on $M$, and  $\tGr_o(p,TM)$ the double cover of $\tGr(p,TM)$ of decomposable unit $p$-vectors.  Let $\Omega^p(M)$ denote the space of smooth $p$-forms on $M$.

Note that, given a $p$-form $\varphi \in \Omega^p(M)$ and $\xi = e_1\wedge\cdots\wedge e_p \in \tGr_o(p,TM)$, $\varphi(\xi) := \varphi(e_1,\ldots,e_p)$ is well-defined.  If  $\varphi$ is closed and $\varphi \le 1$ on $\tGr_o(p,TM)$, then $\varphi$ is a \emph{calibration}.  The condition that $\varphi \le 1$ on $\tGr_o(p,TM)$ is often expressed as $\varphi_{|\xi} \le \tvol_{|\xi}$.  Assume $\varphi$ is a calibration.  Let 
$$
  \tGr(\varphi) \ := \ \{ \xi \in \tGr_o(p,TM) \ | \ \varphi(\xi) =  1 \}
$$
denote the set of \emph{(oriented) calibrated planes}, and $\tGr(\varphi)_x$ the fibre over $x \in M$.  An oriented $p$-dimensional submanifold $N \subset M$ is \emph{calibrated} if $T_xN \in \tGr(\varphi)_x$, for all $x \in N$.  That is, $\varphi_{|N} = \tvol_N$.  Compact calibrated submanifolds have the property that they are globally volume minimizing in their homology classes \cite{HL}.  The first step in the identification or construction of calibrated submanifolds is the determination of $\tGr(\varphi)$.  However, this is often a difficult problem.

Notice that elements of $\tGr(\varphi)_x$ are critical points of $\varphi_x : \tGr_o(p,T_xM)\to \bR$.  However,
it is not the case that every critical point is an element of $\tGr(\varphi)_x$.  (See \S\ref{S:su2eg} below.)  Let $C(\varphi)_x \subset \tGr_o(p,T_xM)$ denote the set of critical points of $\varphi_x$, and $C(\varphi) \subset \tGr_o(p,TM)$ the associated sub-bundle.  An oriented $p$-dimensional submanifold $N \subset M$ is \emph{$\varphi$--critical} if $T_x N \subset C(\varphi)_x$, for all $x \in N$.  While the calibrated submanifolds are prized as volume minimizers in their homology classes, the $\varphi$--critical submanifolds are also interesting.  Unal showed that if the corresponding critical value is a local maximum, then the $\varphi$--critical submanifold is minimal \cite[Th. 2.1.2]{unal}.  I will prove (Theorem \ref{T:minimal}): if $\varphi$ is parallel, then the $\varphi$--critical submanifolds with nonzero critical value are minimal.  I will also show that the $\varphi$-critical submanifolds are characterized by an exterior differential system $\sP$ (Theorem \ref{T:eds}).

\subsection{Contents}\label{S:contents}

We begin in Section \ref{S:basics} with the simple case of a constant coefficient calibration $\phi \in \tw^pV^*$.  In Proposition \ref{P:cp} I identify the critical points $C(\phi) \supset \tGr(\phi)$ as the annihilator of a linear subspace $\Phi \subset \tw^pV^*$.  In the case that $\phi$ is invariant under a Lie subgroup $H \subset \tO(V)$, $\Phi$ is a $H$-submodule of $\tw^pV^*$ (Lemma \ref{L:invariant}).  (Of course, every $\phi$ is invariant under the trivial group $\{ \tId \} \subset \tO(V)$.)  Several examples are discussed in Section \ref{S:eg}, and a vector-product variation of Proposition \ref{P:cp} is given in Proposition \ref{P:prod}.

In Section \ref{S:parallel_calib}, Proposition \ref{P:cp} is generalized to a parallel calibrations on a connected, $n$--dimensional, Riemannian manifold $M^n$.  Given an $n$-dimensional $H$--manifold $M$, a $H$--invariant $\phi \in \tw^pV^*$ naturally defines a parallel $p$-form $\varphi$ on $M$.  Conversely, every parallel $p$-form $\varphi$ on a Riemannian manifold arises in this fashion.  (See \S\ref{S:varphi} for a description of the construction.)  As a parallel form, $\varphi$ is {\it a priori} closed and thus a calibration on $M$.  Similarly, $\Phi$ defines a sub-bundle $\Phi_M \subset \tw^p T^*M$.  Let $\sP \subset \Omega^p(M)$ denote smooth sections of $\Phi_M$.  A $p$--dimensional submanifold $N^p\subset M$ is an \emph{integral submanifold of $\sP$} if $\sP_{|N} = \{0\}$.

\begin{theorem}\label{T:eds}  Assume that $M^n$ is a connected Riemannian manifold, and $\varphi$ a parallel calibration.  A submanifold $N^p$ is $\varphi$--critical if and only if $N$ is an integral manifold of $\sP$.  In particular, every calibrated submanifold of $M$ is an integral manifold of $\sP$.  
\end{theorem}

\noindent 
Proposition \ref{P:prod} (the vector-product variant) easily generalizes to give an alternative formulation of the $\varphi$--critical submanifolds as those submanifolds $N$ with the property that $T_xN$ is closed under an alternating $(p-1)$--fold vector product $\rho : \tw^{p-1} TM \to TM$.  

\medskip

If $N \subset M$ is $\varphi$--critical, then $\varphi_{|N} = \varphi_o \, \tvol_N$, where $\varphi_o$ is a constant.  Refer to this constant as the \emph{critical value of $\varphi$ on $N$}.

\begin{theorem} \label{T:minimal}
Assume that $M$ is a Riemannian manifold,  $\varphi \in \Omega^p(M)$ a parallel calibration, and $N \subset M$ a $\varphi$--critical submanifold.  If the critical value of $\varphi$ on $N$ is nonzero, then $N$ is a minimal submanifold of $M$.
\end{theorem}

\noindent
Theorems \ref{T:eds} and \ref{T:minimal} are proven in Sections \ref{S:varphi} and \ref{S:minimal}, respectively.

Finally in \S\ref{S:I} it is shown that the ideal $\sI \subset \Omega(M)$ algebraically generated by $\sP$ is differentially closed and that, in general, the system fails to be involutive.

\subsection*{Acknowledgements}  I am indebted to R. Harvey for valuable feedback, and bringing connections with \cite{HLpot_thy} to my attention.  I especially appreciate a pointed observation that led me to Lemma \ref{L:spinors}.

\subsection*{Notation}  
Fix index ranges
$$
  i,j \in \{ 1,\ldots,n\} \, , \quad 
  a,b \in \{ 1,\ldots,p \} \, , \quad
  s,t \in \{ p+1,\ldots,n \} \, .
$$
The summation convention holds:  when an index appears as both a subscript and superscript in an expression, it is summed over.

\section{The infinitesimal picture} \label{S:V}

\subsection{The basics} \label{S:basics}
 
Let $\phi \in \tw^pV^*$ and $\xi = e_1\wedge\cdots\wedge e_p \in \tGr_o(p,V)$.  Then $\phi(\xi) = \phi(e_1,\ldots,e_p)$ is a well-defined function on $\tGr_o(p,V)$.  Fix a nonzero $\phi \in \tw^p V^*$, with the property that $\mathrm{max}_{\tGr_o(p,V)} \, \phi = 1$.  The set of (oriented) \emph{calibrated $p$-planes} is  
\begin{displaymath}
  \tGr(\phi) \ := \ \left\{ \xi \in \tGr_o(p,V) \ | \ \phi(\xi) = 1 \right\} \, .
\end{displaymath}
Let $C(\phi)\subset\tGr_o(p,V)$ denote the critical points of $\phi$.  Then 
$$
  \tGr(\phi) \ \subset \ C(\phi) \, .
$$

Let $\cF_V$ denote the set of orthonormal bases (or frames) of $V$.  Given $e = (e_1 , \ldots , e_n) \in \cF_V$, let $e^* = (e^1,\ldots,e^n)$ denote the dual coframe.  Then
\begin{displaymath}
  \phi \ = \ \phi_{i_1\cdots i_p}  e^{i_1} \wedge \cdots \wedge e^{i_p} \, , 
\end{displaymath}
uniquely determines functions $\phi_{i_1\cdots i_p}$, skew-symmetric in the indices, on $\cF_V$.  Note that $|\phi_{i_1\cdots i_p}| \le 1$, and $\xi = e_{i_1}\wedge\cdots\wedge e_{i_p} \in \tGr(\phi)$ if and only if equality holds.  

Next we compute $\td\phi_{|\xi}$.  Let $\tO(V)$ denote the Lie group of linear transformations $V\to V$ preserving the inner product, and let $\fo(V)$ denote its Lie algebra.  Let $\theta$ denote the $\fo(V)$--valued Maurer-Cartan form on $\cF_V$: at $e \in \cF_V$, $\theta_e = \theta^j_k \, e_j \ot e^k$, where the coefficient 1-forms $\theta^j_k = -\theta^k_j$ are defined by $\td e_j = \theta_j^k \, e_k$.  Then $\{ \theta^i_j \ | \ i < j \}$ is a basis for the 1-forms on $\cF_V$.  

If $\xi = e_{i_1}\wedge\cdots\wedge e_{i_p}$ is viewed as a map $\cF_V \to \tGr_o(p,V)$, then 
$$
  \td \xi \ = \ \sum_{1\le a\le p} 
 e_{i_1}\wedge\cdots\wedge e_{i_{a-1}}\wedge \,  \theta_{i_a}^k e_k \, \wedge
  e_{i_{a+1}}\wedge\cdots\wedge e_{i_p} \, .
$$
Thus
\begin{eqnarray*}
   \td\phi_{\xi} & = & \td \phi(e_{i_1},\ldots,e_{i_p}) \\
   & = & \sum_{1\le a\le p} \phi( e_{i_1} , \ldots , e_{i_{a-1}} ,\, \theta_{i_a}^k
   e_k ,\, e_{i_{a+1}} , \ldots , e_{i_p} ) \\
   & = & \sum_{1\le a\le p} \theta_{i_a}^k \,\phi( e_{i_1} , \ldots , e_{i_{a-1}} , 
   e_k , e_{i_{a+1}} , \ldots , e_{i_p} ) \\
   & = & \sum_{1\le a\le p} 
   \phi_{i_1\cdots i_{a-1} k i_{a+1}\cdots i_p} \, \theta^k_{i_a} \, .
\end{eqnarray*}
The skew-symmetry of $\phi$ and $\theta$ imply that $\phi_{i_1\cdots i_{a-1} k i_{a+1}\cdots i_p} \, \theta^k_{i_a}$ vanishes if $k \in \{ i_1,\ldots,i_p\}$.  The $\{ \theta^k_{i_a} \ | \ 1\le a\le p \, , \ k \not\in \{i_1,\ldots,i_p\} \}$ are linearly independent on $\cF_V$, and may be naturally identified with linearly independent 1-forms on $\tGr_o(p,V)$ at $\xi$.  Consequently, $\td \phi_\xi = 0$, and \begin{equation} \label{E:cp}
  \xi = e_{i_1}\wedge\cdots\wedge e_{i_p} \hbox{ is a critical point if and only if }
  \phi_{i_1\cdots i_{a-1} k i_{a+1}\cdots i_p} \, \theta^k_{i_a} \ = \ 0 \, .
\end{equation}  
An equivalent, index-free formulation of this observation is given by the lemma below. 

\begin{lemma} \label{L:cp1}
  A $p$-plane $\xi$ is a critical point of $\phi$ if and only if $(v \lefthook\phi)_{|\xi} = 0$ for all $v \in \xi^\perp$.
\end{lemma}

\begin{remark*}
The lemma was first observed by Harvey and Lawson (cf. Remark on page 78 of {HL}), and is often referred to as the First Cousin Principle.
\end{remark*}

The lemma allows us to characterize the critical points $\xi \in \tGr_o(p,V)$ of $\phi$ as the $p$-planes on which a linear subspace $\Phi \subset \tw^p V^*$ vanishes.  Forget, for a moment, that $\theta$ is a 1-form on $\cF_V$ and regard it simply as an element of $\fo(V)$.  Let $\theta.\phi$ denote the action of $\theta$ on $\phi$.  The action yields a map $\sfP: \fo(V) \to \tw^p V^*$ sending $\theta \mapsto \theta.\phi$.  Define
\begin{displaymath}
  \Phi \ := \ \sfP(\fo(V)) \ \subset \ \tw^pV^* \, .
\end{displaymath} 
Notice that the $e^{i_1}\wedge\cdots\wedge e^{i_p}$--coefficient of $\theta . \phi$ is $\phi_{i_1\cdots i_{a-1} k i_{a+1}\cdots i_p} \, \theta^k_{i_a}$.  From this observation, \eqref{E:cp}, and the fact that the Maurer-Cartan form $\theta_e : T_e\cF_V \to \fo(V)$ is a linear isomorphism, we deduce the following.

\begin{proposition} \label{P:cp}
  The set of $\phi$--critical planes is $C(\phi) = \tGr_o(p,V) \, \cap \, \tAnn(\Phi)$.
\end{proposition}
\medskip

\noindent{\it Remark.}  The map $\sfP$ is the restriction of the map $\lambda_\phi : \mathrm{End}(V) \to \tw^pV^*$ in \cite{HLpot_thy} to $\fo(V)$.
Corollary 2.6 of \cite{HLpot_thy} is precisely the observation that elements of $\Phi$ vanish on $\tGr(\phi) \subset C(\phi)$.  Indeed, Proposition \ref{P:cp} above follows from Proposition A.4 of that paper.  This is seen by observing that if $A \in \fo(V) \subset \mathrm{End}(V)$, then $\mathrm{tr}_\xi A = 0$.  Then their (A.2) reads $\lambda_\phi(A)(\xi) = \phi(D_{\widetilde A} \xi)$.  It now suffices to note that their $\{ \lambda_\phi(A) \ | \ A \in \fo(V) \}$ is our $\Phi$, and that $\{ D_{\widetilde A} \xi \ | \ A \in \fo(V) \} = T_\xi \tGr_o(p,V)$.
\medskip

\noindent{\it Remark.}  Each $\phi \in \tw^pV^*$ naturally determines an alternating $(p-1)$-fold vector product $\rho$ on $V$.  An equivalent formulation of Proposition \ref{P:cp} is given by Proposition \ref{P:prod} which asserts that $\xi \in C(\phi)$ and only if $[\xi] \in \tGr(p,V)$ is $\rho$--closed.  


\section{Examples and the product characterization} \label{S:eg}

\subsection{Invariant forms}\label{S:module}

Let $G$ denote the stabilizer of $\phi$ in $\tO(V)$.  Many of the calibrations that we are interested in have a nontrivial stabilizer; but, of course, all statements hold for trivial $G$.  Observe that $\Phi$ is a $\fg$-module.  This is seen as follows.  Let $\fg$ denote the Lie algebra of $G$.  As a $\fg$-module $\fo(V)$ admits a decomposition of the form $\fo(V) = \fg \op \fg^\perp$.  By definition, the kernel of $\sfP$ is $\fg$.  In particular, $\Phi = \sfP(\fg^\perp)$.  It is straightforward to check that $\sfP$ is $G$-equivariant, and we have the following lemma.
\begin{lemma}\label{L:invariant}
  The subspace $\Phi = \sfP(\fg^\perp) \subset \tw^pV^*$ is isomorphic to 
  $\fg^\perp$ as a $G$-module.
\end{lemma}

\noindent

Below I identify $\Phi$ for some well-known examples.  The calibrations $\phi$ and characterizations of $\tGr(\phi)$ in \S\ref{S:ass}--\ref{S:sLag} were introduced in \cite{HL}.  

\subsection{Associative calibration}\label{S:ass}

Consider the standard action of the exceptional $G=G_2$ on the imaginary octonions $V = \tIm \bO = \bR^7$.  As a $G_2$--module the third exterior power decomposes as $\tw^3V^* = \bR \op  V^3_{1,0} \op  V^3_{2,0}$.  (Cf. \cite[Lemma 3.2]{FG} or \cite[p. 542]{BrExHol}.)  Here $V^3_{1,0} = V$ as $G_2$--modules.  The trivial subrepresentation $\bR \subset \tw^3V^*$ is spanned by an invariant 3-form $\phi$, the associative calibration.  It is known that $\xi \in \tGr(\phi)$ if and only if the forms $V^3_{1,0} = \{ \ast(\phi\wedge\a) \ | \ \a \in V^* \}$ vanish on $\xi$ \cite[Corollary 1.7]{HL}.  Here $*(\phi\wedge\alpha)$ denotes the Hodge star operation on the 4-form $\phi\wedge\alpha$.  As $\Phi = V^3_{1,0}$, we have $C(\phi) = \tGr(\phi)$.  

\subsection{Coassociative calibration}\label{S:coa}

Again we consider the standard action of $G_2$ on $V=\tIm\bO = V_{1,0}$.  The Hodge star commutes with the $G_2$ action.  So the fourth exterior power decomposes as $\tw^4V^* = V^4_{0,0} \op  V^4_{1,0} \op  V^4_{2,0}$, with $V^4_{a,b} = \ast V^3_{a,b}$.  The trivial subrepresentation is spanned by the invariant coassociative calibration $\ast\phi$.  A 4-plane $\xi$ is calibrated by $\ast\phi$ if and only if $\phi_{|\xi} \equiv 0$ \cite[Corollary 1.19]{HL}.  Equivalently, the 4-forms of $V^4_{1,0} = \{ \phi\wedge\a \ | \ \a \in V^* \}$ vanish on $\xi$.  As $\Phi = V^4_{1,0}$, we again have $C(\phi) = \tGr(\phi)$.

\subsection{Cayley calibration}\label{S:cay}

Consider the standard action of $G = B_3 = \tSpin(7) \subset \tSO(8)$ on the octonions $V=\bO=\bR^8$.  The fourth exterior power decomposes as  $\tw^4 V^*=  V^4_{0,0,0} \op  V^4_{1,0,0} \op  V^4_{2,0,0} \op  V^4_{0,0,2}$.  (Cf. \cite[p. 548]{BrExHol} or \cite[Lemma 3.3]{MR859598}.)  The trivial subrepresentation $V^4_{0,0,0}$ is spanned by the invariant, self-dual Cayley 4-form $\phi = \ast\phi$.  It is known that $\xi \in \tGr(\phi)$ if and only if the forms $V^4_{1,0,0} = \{ \a.\phi \ | \ \a \in  V^2_{1,0,0} \}$ vanish on $\xi$ \cite[Proposition 1.25]{HL}; here $ V^2_{1,0,0}  = \{ \a \in \tw^2 V^* \ | \ \ast(\a\wedge\phi) = 3\, \alpha \} \simeq \fg^\perp$.  As $\Phi = V^4_{1,0,0}$, we have $C(\phi) = \tGr(\phi)$.

\subsection{Special Lagrangian calibration}\label{S:sLag}

Regard $V := \bC^{m}$ as a real vector space.  Given the standard coordinates $z = x+\ti y$, 
$$
  V^* = \tspan_\bR \left\{ \half(\dz + \dbz) \, , \, 
                      -\ihalf(\dz - \dbz) \right\} \, .
$$ 
Set 
\begin{eqnarray*}
  \sigma & = & - \tfrac{\ti}{2} \left( \dz^1 \wedge \dbz^1 + \cdots + 
           \dz^m \wedge \dbz^m \right) \, , \\
  \Upsilon & = & {}\td z^1 \wedge \cdots \wedge \td z^m \, .
\end{eqnarray*}
The special Lagrangian calibration is $\tRe \Upsilon$.  An $m$-dimensional submanifold $i : M \to V$ is calibrated if and only if $i^* \sigma = 0 = i^* \tIm \Upsilon$.  (Recall that $i^*\sigma = 0$ characterizes the $m$-dimensional Lagrangian submanifolds.)

The special Lagrangian example is distinct from those above in that 
$$
  \fsu(m)^\perp \ = \  \bR \op W \ \subset \ \tw^2 V
$$
is reducible as an $\fsu(m)$--module.  The trivial subrepresentation is spanned by $\sigma$.  

The $\fsu(m)$ module $\Phi$ decomposes as $\Phi_0 \op \Phi_W$, where $\Phi_0 = \tspan_\bR \{ \tIm\Upsilon \}$ and $\Phi_W = W.(\tRe\Upsilon)$.  The elements of the sub-module $\Phi_W$ may be described as follows.  Let $J \subset \{1,\ldots,m\}$ be a multi-index of length $|J|=\ell$, and $\dz^J := dz^{j_1}\wedge\cdots\wedge\dz^{j_{\ell}}$.  The reader may confirm that $\Phi_W = \tspan_\bR\{ \tRe \, \dz^J \wedge \sigma \, , \ \tIm \, \dz^J \wedge \sigma \ : \ |J|=m-2 \}$.  

In the remark of \cite[p.90]{HL} Harvey and Lawson showed that an $m$-plane $\zeta$ is Lagrangian if and only if the forms $\Psi := \{ \td z^J \wedge \sigma^p \ : \ 2p+|J| = m \, , \ p>0 \} \supset \Phi_W$ vanish on $\zeta$.   So $\pm\xi \in \tGr(\tRe\Upsilon)$ if and only if $\tIm\Upsilon_{|\xi} = 0 = \Psi_{|\xi}$, while $\xi \in C(\tRe\Upsilon)$ if and only if $\tIm\Upsilon_{|\xi} = 0 = \Phi_{W}{}_{|\xi}$.  So it seems \emph{a priori} that a critical $\xi$ need not be calibrated.  Nonetheless Zhou \cite[Theorem 3.1]{zhou} has shown that $\pm\tGr(\tRe\Upsilon) = C(\tRe\Upsilon)$.

\subsection{Squared spinors.}\label{S:spinors}

In \cite{DH} Dadok and Harvey construct calibrations $\phi \in \tw^{4p}V^*$ on vector spaces of dimension $n=8m$ by squaring spinors.  Let me assume the notation of that paper: in particular, $\bP = \bS^+ \op \bS^-$ is the decomposition of the space of pinors into positive and negative spinors, $\e$ an inner product on $\bP$, and $\tCl(V) \simeq \tEnd_\bR(\bP)$ the Clifford algebra of $V$.  Given $x,y,z \in \bP$, $x\circ y \in\tEnd_\bR(\bP)$ is the linear map $z \mapsto \e(y,z) x$.

Given a unit $x \in \bS^+$, $\underline{\phi} = 16^m x \circ x \in \tEnd_\bR(\bS^+) \subset \tEnd_\bR(\bP)$ may be viewed as an element of $\tw V^* \simeq \tCl(V)$.  Let $\phi_k \in \tw^k V^*$ be the degree $k$ component of $\underline{\phi}$.  Each $\phi_k$ is a calibration, and $\phi_k$ vanishes unless $k = 4p$.  (Also, $\phi_0 = 1$ and $\phi_n = \mathit{vol}_V$.)
The Cayley calibration of \S\ref{S:cay} is an example of such a calibration; see \cite[Prop. 3.2]{DH}.

Given such a calibration $\phi = \phi_{4p}$, Dadok and Harvey construct $4p$-forms $\Psi_1, \ldots , \Psi_N$, $N = \frac{1}{2}(16)^m-1$, that characterize $\tGr(\phi)$; that is, $\xi \in \tGr(\phi)$ if and only if $\Psi_j(\xi) = 0$ \cite[Th. 1.1]{DH}.  

\begin{lemma}\label{L:spinors}
The span of the $\Psi_j$ is our $\Phi$.  In particular, $C(\phi) = \tGr(\phi)$.
\end{lemma}

\begin{proof}
Continuing to borrow the notation of \cite{DH}, the proof may be sketched as follows.  Complete $x = x_0$ to an orthogonal basis $\{x_0,x_1,\ldots, x_N\}$ of $\bS^+$.  Then $\Psi_j$ is the degree $4p$ component of $16^m x_j \circ x_0 \in \tEnd_\bR(\bS^+) \subset \tw V^*$.
Our $\Phi$ is spanned by $\gamma_j$, the degree $4p$ component of $16^m(x_j \circ x_0 + x_0 \circ x_j)$.  Let  $\langle x\circ y , \xi \rangle$ denote the extension of the inner product on $V$ to $\tEnd_\bR(\bP) \simeq \tCl(V) \simeq \tw V^*$.  (See \cite{DH}.)  Given $\xi \in \tGr_o(4p,V)$, 
\begin{eqnarray*}
  \Psi_j(\xi) & = & 16^m \langle x_j \circ x_0 , \xi \rangle \\
  \gamma_j(\xi) & = & 16^m \langle x_j \circ x_0 + x_0 \circ x_j , 
    \xi \rangle \, .
\end{eqnarray*}
To see that $\Phi = \tspan\{\Psi_1 ,\ldots,\Psi_N\}$ it suffices to note that
\begin{displaymath}
  16^m \langle x_0 \circ x_j , \xi \rangle = \e(x_0 , \xi x_j) = 
  \e(x_j , \xi x_0 ) = 16^m \langle x_j \circ x_0 , \xi \rangle \, , 
\end{displaymath}
when $\xi \in \tw^{4p}V^*$.  Hence $\gamma_j = 2 \Psi_j$.
\end{proof}
\medskip

\noindent{\it Remark.}  Zhou showed that $C(\phi) = \tGr(\phi)$ for many well known calibrations \cite{zhou}.  As the following example illustrates, this need not be the case.

\subsection{Cartan 3-form on $\fg$.}\label{S:su2eg}

Let $G$ be a compact simple Lie group with Lie algebra $\fg$.  Set $V = \fg$ and consider the adjoint action.  Every simple Lie algebra admits an (nonzero) invariant 3-form, the Cartan form $\phi$, defined as follows.  Given $u,v \in \fg$, let $[u,v] \in \fg$ and $\langle u,v\rangle \in \bR$ denote the Lie bracket and invariant inner product, respectively.  Then $\phi(u,v,w) = c \langle u , [v,w] \rangle$, with $\frac{1}{c}$ the length of a highest root $\delta$.  It is immediate from Lemma \ref{L:cp1} that $\xi$ is a critical point if and only if $\xi$ is a subalgebra of $\fg$.  

\begin{proposition}\label{P:3subalg}
  A 3-plane $\xi$ is $\phi$-critical if and only if it is a subalgebra of $\fg$.  
\end{proposition}

\noindent\emph{Remark.}  The proposition generalizes to arbitrary $\phi$.  See Proposition \ref{P:prod}.

\medskip

The $\fsu(2)'s$ in $G(3,\fg)$ corresponding to a highest root all lie in the same $\tAd(G)$-orbit and Tasaki \cite{Tasaki1985} showed that this orbit is $\tGr(\phi)$.  (Thi \cite{Thi1979} had observed that the corresponding $\tSU(2)$ are volume minimizing in their homology class in the case that $G = \tSU(n)$.)  If the rank of $\fg$ is greater than 1, then $\fg$ contains 3-dimensional subalgebras that are not associated to a highest root.  Thus $\tGr(\phi) \subsetneqq C(\phi)$. 

\medskip

\noindent{\it Remark.}  The quaternionic calibration on $\mathbb{H}^n$ also satisfies $\tGr(\phi) \subsetneqq C(\phi)$; see \cite{unal} for details.

\subsection{Product version of Proposition \ref{P:cp}}\label{S:prod}

Proposition \ref{P:3subalg} asserts that a 3-plane $\xi$ is $\phi$-critical, $\phi$ the Cartan 3-form, if and only if $\xi$ is closed under the Lie bracket.  This is merely a rephrasing of Proposition \ref{P:cp}, and an analogous statement holds for any calibration.

Given a $p$-form $\phi \in \tw^p V^*$, define a $(p-1)$--fold alternating vector product $\rho$ on $V$ by 
\begin{equation} \label{E:prod}
  \phi(u,v_2,\ldots,v_p) \ =: \ \langle u , \rho(v_2,\ldots,v_p) \rangle \, .
\end{equation}

\medskip

\noindent\emph{Example.}
In the case that $V = \fg$ and $\phi$ is the Cartan 3-form, $\rho$ is a multiple of the Lie bracket.  

\medskip

\noindent
The following proposition is a reformulation of Lemma \ref{L:cp1}.

\begin{proposition} \label{P:prod}
  Let $\phi \in \tw^pV^*$, and let $\rho$ denote the associated $(p-1)$--fold 
  alternating product defined in \eqref{E:prod}.  Then a $p$-plane 
  $\xi \in \tGr_o(p,V)$ is $\phi$--critical if and only if $\xi$ is 
  $\rho$--closed. 
\end{proposition}

\noindent\emph{Example.}  When $V = \bO$ and $\phi$ is the Cayley calibration, then $\rho$ is a multiple of the triple cross product.  See \cite[\S IV.1.C]{HL} where it is shown that a 4-plane is Cayley if and only if it is closed under the triple cross product.

\medskip

Note that  
\begin{equation}\label{E:orth}
  \rho(v_2,\ldots,v_p) \ \hbox{ is orthogonal to } \ v_2,\ldots,v_p \, . 
\end{equation}
In particular, $\rho$ may be viewed as a generalization of Gray's vector cross product, satisfying \cite[(2.1)]{MR0243469} but not necessarily \cite[(2.2)]{MR0243469}.  

Assume that $\xi = e_1\wedge\cdots\wedge e_p\in C(\phi)$.  Then \eqref{E:orth} and Proposition \ref{P:prod} imply $\rho(e_2,\ldots,e_p) = \phi(\xi) \, e_1$.  This yields the following.

\begin{corollary}\label{C:prod=0}
 Let $\xi \in \tGr_o(p,V)$.  The product $\rho$ vanishes on $[\xi] \in \tGr(p,V)$ 
 if and only if $\xi \in C(\phi)$ and $\phi(\xi)=0$.
\end{corollary}

\section{Parallel calibrations} \label{S:parallel_calib}

\subsection{Orthonormal coframes on \boldmath$M$\unboldmath} \label{S:F}

Let $V$ be an $n$-dimensional Euclidean vector space.  Let $M$ be an $n$-dimensional connected Riemannian manifold, and let $\pi : \cF \to M$ denote the \emph{bundle of orthogonal coframes}.  Given $x\in M$, the elements of the fibre $\pi^{-1}(x)$ are the linear isometries $u : T_x M \to V$.  Given $g \in \tO(V)$, the right-action $u \cdot g := g^{-1} \circ u$ makes $\cF$ a principle right $\tO(V)$--bundle.

The canonical $V$--valued 1-form $\w$ on $\cF$ is defined by 
$$
  \w_u(v) \ := \ u( \pi_* v ) \, , 
$$ 
$v \in T_u\cF$.  Let $\vartheta$ denote the unique torsion-free, $\fo(V)$--valued connection 1-form on $\cF$ (the Levi-Civita connection form).  Fix an orthonormal basis $\{ \mathsf{v}_1,\ldots,\mathsf{v}_n \}$ of $V$.  Then we may define 1-forms $\w^i$ on $\cF$ by 
$$
  \w_u \ =: \ \w_u^i \, \mathsf{v}_i \, .
$$
Let $\mathsf{v}^1,\ldots,\mathsf{v}^n$ denote the dual basis of $V^*$, and define $\vartheta^i_j$ by $\vartheta = \vartheta^i_j \, \mathsf{v}_i \ot \mathsf{v}^j$.  Then
$$
  \vartheta^i_j \, + \, \vartheta^j_i \ = \ 0 \, , \quad \mathrm{and} \quad
  \td \w^i \ = \ -\vartheta^i_j \wedge \w^j \, .
$$

Given $u \in \cF$, let $\{ e_1 , \ldots , e_n \}$, $e_i = e_i(u) := u^{-1}(\sfv_i)$, denote the corresponding orthonormal basis of $T_xM$.

\subsection{\boldmath$H$--manifolds\unboldmath} \label{S:Hmfd}

Suppose $H \subset \tO(V)$ is a Lie subgroup.  If the bundle of orthogonal coframes over $\cF \to M$ admits a sub-bundle $\cE \to M$ with fibre group $H$, then we say $M$ carries a \emph{$H$--structure}.  The $H$--structure is \emph{torsion-free} if $\cE$ is preserved under parallel transport by the Levi-Civita connection in $\cF$.  In this case we say $M$ is a \emph{$H$--manifold}.

When pulled-back to $\cE$, the forms $\w^i$ remain linearly independent, but $\vartheta$ takes values in the Lie algebra $\fh \subset \fo(V)$ of $H$.

\subsection{The construction of \boldmath$\varphi$ and $\Phi_M$\unboldmath} \label{S:varphi}

I now prove Theorem \ref{T:eds}.  Assume that $M$ is a $H$--manifold.  Let $\pi_* : T_u \cE \to T_x M$ denote the differential of $\pi:\cE\to M$.  Any $\phi \in \tw^p V^*$ induces a $p$-form $\varphi$ on $\cE$ by $\varphi_u(v_1,\ldots,v_p) = \phi( \w_u(v_1) , \ldots , \w_u(v_p) )$.  Assume $\phi$ is $H$--invariant.  Then $\varphi$ descends to a well-defined $p$-form on $M$.  Since $\cE\subset\cF$ is preserved under parallel transport, $\varphi$ is parallel and therefore closed.  Conversely, every parallel $p$-form $\varphi$ arises in such a fashion: fix $x_o \in M$, and take $V = T_{x_o}M$ and $\phi = \varphi_{x_o}$.

Assume that $\mathrm{max}_{\tGr_o(p,V)}\phi = 1$.  Then $\varphi$ is a calibration on $M$.

Since $H$ is a subgroup of the stabilizer $G$ of $\phi$, Lemma \ref{L:invariant} implies $\Phi \subset \tw^pV^*$ is a $H$--module.  It follows that $\Phi$ defines a sub-bundle $\Phi_M \subset \tw^pT^*M$.  Explicitly, given $u \in \cE_x$, $\Phi_{M,x} := (u^{-1})^*(\Phi) \subset \tw^pT^*_xM$.  The fact that $\Phi$ is an $H$--module implies that the definition of $\Phi_{M,x}$ is independent of our choice of $u \in \cE_x$.  

Let $\sP \subset \Omega^p(M)$ denote space of smooth sections of $\Phi_M$.  Theorem \ref{T:eds} now follows from Proposition \ref{P:cp}.

\medskip

\noindent\emph{Remark.}  Note that Proposition \ref{P:prod} also extends to parallel calibrations in a straightforward manner.

\subsection{Proof of Theorem \ref{T:minimal}}  \label{S:minimal}

Recall the notation of Section \ref{S:F}; in particular the framing $e = e(u)$ associated to $u \in \cF$.  Given a $p$-form $\psi \in \Omega^p(M)$, define functions $\psi_{i_1\cdots i_p} : \cF \to \bR$ by $\psi_{i_1\cdots i_p}(u) := \psi(e_{i_1},\ldots,e_{i_p})$.  The fact that $\varphi$ is parallel implies
\begin{equation} \label{E:dvarphi}
  \td \varphi_{i_1\cdots i_p} \ = \ ( \vartheta . \varphi)_{i_1\cdots i_p} \, ,
\end{equation}
where $\vartheta . \varphi$ denotes the $\fo(n)$--action of $\vartheta$ on $\varphi$. 

The following notation will be convenient.  Let $\{ i_1 , \ldots , i_m \} \subset \{1,\ldots,n\}$ and $\{a_1,\ldots,a_m\} \subset \{1,\ldots,p\}$.  If the $\{a_1,\ldots,a_m\}$ are pairwise distinct, then let $\psi^{a_1\cdots a_m}_{i_1\cdots i_m}$ denote the function obtained from $\psi_{12\cdots p}$ by replacing the indices $a_\ell$ with $i_\ell$, $1\le \ell \le m$.  Otherwise, $\psi^{a_1\cdots a_m}_{i_1\cdots i_m} = 0$.  For example, $\psi^2_s = \psi_{1s3\cdots p}$ and $\psi^{13}_{st} = \psi_{s2t4\cdots p}$.  Note that $\psi^{a_1\cdots a_m}_{i_1\cdots i_m}$ is skew-symmetric in both the upper indices and the lower indices; for example, $\psi^{abc}_{rst} = -\psi^{bac}_{rst} = -\psi^{abc}_{tsr}$.

Define
$$
  \cC \ := \ \{ u \in \cF \ | \ e_1\wedge\cdots\wedge e_p \in C(\varphi_x) \, , 
  \ x = \pi(u) \, , \ e = e(u) \} \, .
$$
It is a consequence of Lemma \ref{L:cp1} that 
$$
  \cC \ = \ \{ u \in \cF \ | \ \varphi^a_s(u) = 0 \ \forall \ 
     1\le a\le p < s \le n \} \, . 
$$

Given a $p$--dimensional submanifold $N \subset M$, a \emph{local adapted framing of $M$ on $N$} is a section $\sigma : U \to \cF$, defined on an open subset $U \subset N$ with the property that $\tspan\{ e_1(x) , \ldots , e_p(x) \} = T_xN \subset T_xM$, $e_a(x) := e_a \circ \sigma(x)$, for all $x \in U$.  When pulled-back to $\sigma(U)$, 
\begin{equation} \label{E:Neds}
  \w^s \ = \ 0 \quad \forall \ p < s \le n \quad \hbox{ and } \quad 
  \w^1\wedge\cdots\wedge\w^p \ \not= \ 0 \, .
\end{equation}
Conversely every $p$--dimensional integral submanifold $\tilde U \subset \cF$ of \eqref{E:Neds} is locally the image $\sigma(U)$ of an adapted framing over a $p$--dimensional submanifold $U \subset M$.

Given $N$, let $\cF_N \subset \cF$ denote the bundle of adapted frames of $M$ over $N$.  As noted above $\w^s{}_{|\cF_N} = 0$.  Differentiating this equation and an application of Cartan's Lemma yields 
$$
  \theta^s_a \ = \ h^s_{ab} \, \w^a
$$
for functions $h^s_{ab} = h^s_{ba} : \cF_N \to \bR$.  The $h^s_{ab}$ are the coefficients of the \emph{second fundamental form of $N \subset M$}.

Observe that $N$ is $\varphi$--critical if and only if $\cF_N \subset \cC$.  Assume that $N$ is $\varphi$--critical.  Then $\varphi^a_s = 0$ on $\cF_N$.  Differentiating this equation yields $0 = \td \varphi^a_s = (\vartheta.\varphi)^a_s = \varphi_o \, \vartheta^a_s + \varphi^{ab}_{st} \, \vartheta^t_b$, where
$$
  \varphi_o \ := \ \varphi_{12\cdots p} \ = \ \varphi(e_1,\ldots,e_p) \, 
$$
is the (constant) critical value of $\varphi$ on $N$.
Equivalently, $\varphi_o \, h^s_{ac} = \varphi^{ab}_{st} \, h^t_{bc}$.  Recalling that $\varphi^{ab}_{st}$ is skew-symmetric and $h^s_{ab}$ is symmetric in the indices $a,b$ yields $\sum_a \varphi_o \, h^s_{aa} = \varphi^{ab}_{st} \, h^t_{ab} = 0$.  If $\varphi_o \not=0$, then $\sum_a h^s_{aa} = 0$ and $N$ is a minimal submanifold of $M$.  This establishes Theorem \ref{T:minimal}.

\medskip

\noindent\emph{Remark.}  Note that a $\varphi$--critical submanifold with $\varphi_o = 0$ need not be minimal.  As an example, consider $M = \bR^n$ with the standard Euclidean metric and coordinates $x = (x^1,\ldots,x^n)$, $n \ge 4$.  The form $\varphi = \td x^1 \wedge \td x^2$ is a parallel calibration on $M$.  Any $2$--dimensional $N \subset \{ x^1 = x^2 = 0 \}$ is $\varphi$--critical with critical value $\varphi_o = 0$, but in general will not be a minimal submanifold of $\bR^n$.

\section{The system $\sP$} \label{S:I}

\subsection{The ideal \boldmath$\sI = \langle \sP \rangle$\unboldmath}
Let $\sI\subset \Omega(M)$ be the ideal (algebraically) generated by $\sP$. 

\begin{lemma*} \label{L:dI}
  The ideal $\sI$ is differentially closed.  That is, $\td \sI \subset \sI$.
\end{lemma*}

\begin{proof}
Let $\vartheta$ be the $\fh$-valued, torsion-free connection on $M$.  Let $\{ u^1 , \ldots , u^n \}$ be a local $H$-coframe.  Note that the coefficients $\varphi_{i_1i_2\cdots i_p}$ of $\varphi$ with respect to the coframe are constant.  The space $\Phi_M$ is spanned by forms of the form $\{ \gamma = \theta.\varphi \ | \ \theta \in \fg^\perp \subset \fh^\perp \}$.  In particular, the coefficients of these spanning $\gamma$ are also constant.  Consequently the covariant derivative is $\nabla \gamma = \vartheta.\gamma$.  Since $\vartheta$ is $\fh$--valued and $\Phi$ is $\fh$--invariant, $\nabla \gamma$ may be viewed as a 1-form taking values in $\Phi_M$.  As the exterior derivative $\td \gamma$ is the skew-symmetrization of the covariant derivative $\nabla \gamma$, it follows that $\td \gamma \in \sI$.
\end{proof}

\subsection{Involutivity}
This section assumes that reader is familiar with exterior differential systems.  Excellent references are \cite{BCG3} and \cite{IL}.

In general, the exterior differential system defined by $\sI$ will fail to be involutive.  In fact, involutivity always fails when $p > \half n$.  This is seen as follows.  Let $\sI^k = \sI \, \cap \, \Omega^k(M)$.  Note that $\sI^a = \{0\}$, for $a < p$.  Let $\sV_k(\sI) \subset \tGr(k,TM)$ denote the $k$-dimensional integral elements $E$ of $\sI$.  Then, 
$$
  \sV_a(\sI) \ = \ \tGr(a,TM) \, , \ \forall \ a < p \, , \quad 
  \hbox{and} \quad
  \sV_p(\sI) \ = \ \{ [\xi] \ | \ \xi \in C(\varphi)  \} \, .
$$

Let $\sV_k(\sI)_x \subset \tGr(k,T_xM)$ denote the fibre over $x \in M$.  Given an integral element $E \in \sV_k(\sI)_x$ spanned by $\{e_1,\ldots,e_k\}\subset T_xM$, the \emph{polar space} of $E$ is 
$$
  H(E) \ := \ \{ v \in T_xM \ | \ \psi(e_1,\ldots,e_k,v) = 0 \, , 
                             \ \forall \, \psi \in \sI^{k+1} \} \ \supset \ E \, .
$$

Suppose that $E_p = [\xi] \in \sV_p(\sI)_x$ .  Let $\{ e_1 , \ldots , e_p \}$ be an orthonormal basis of $E$ and set 
$E_a = \tspan \{ e_1 , \ldots , e_a \}$, $1 \le a \le p$.  Since $\sI^a = \{0\}$, $a < p$,  we have $H(E_a) = T_xM$ and $c_a := \tcodim \, H(E_a) = 0$ for $1 \le a \le p-2$.  

Note that $0\not= v \in H(E_{p-1}) \backslash E_{p-1}$ if and only if $\{ v , e_1,\ldots,e_{p-1} \}$ spans a $\varphi$--critical plane. Proposition \ref{P:prod} implies that the span of $\{v,e_1,\ldots,e_{p-1}\}$ is closed under the product $\rho$.  Suppose that $\varphi_o = \varphi(\xi) = \varphi(e_1,\ldots,e_p) \not=0$.  Then \eqref{E:orth} implies $\rho(e_1,\ldots , e_{p-1}) = \phi(E) \, e_p \not=0$, and this forces $H(E_{p-1}) = E$.  So $c_{p-1} := \tcodim \, H(E_{p-1}) = n-p$.  Cartan's Test (cf. \cite[Theorem 7.4.1]{IL} or \cite[Theorem III.1.11]{BCG3}) implies that 
\begin{equation} \label{E:codim}
  \mathrm{codim}_E \sV_p(\sI) \ \ge \ n-p \, .
\end{equation}

Note that the Hodge dual $\ast\varphi \in \Omega^{n-p}$ is also a parallel calibration on $M$; the associated ideal is $\ast\sI$, the Hodge dual of $\sI$.  In particular $\sV_{n-p}(\ast\sI) = \{ E^\perp \ | \ E \in \sV_p(\sI) \}$, so that $\mathrm{codim}_{E^\perp} \sV_{n-p}(\ast\sI) = \mathrm{codim}_E \sV_p(\sI)$.  It follows that equality fails in \eqref{E:codim} when $p > \half n$: the system $\sI$ is not involutive.

\medskip

\noindent\emph{Remark.}  For example, $\sI$ fails to be involutive in the case that $M$ is a $G_2$--manifold and $\varphi$ is the coassociative calibration of \S\ref{S:coa}.  Here, $n=7$ and $p=4$, so that $n-p = 3$, while $\tcodim_E \sV_4(\sI) = 4$.  It fact, $\sP = \{ \alpha \wedge (\ast\varphi) \ | \ \alpha \in \Omega^1(M) \}$, where $\ast\varphi \in \Omega^3(M)$ is the associative calibration.  As is well-known, coassociative submanifolds are integral manifolds of $\{ \ast\varphi = 0 \}$, and this system is involutive.

\medskip

\noindent\emph{Remark.}  If the critical value $\varphi_o = \varphi(\xi)$ equals zero, then Corollary \ref{C:prod=0} implies that the $\rho$ vanishes on $E$.  In this case $H(E_{p-1}) = \{ v \in T_xM \ | \ \rho(v,a_1,\ldots,a_{p-2}) = 0 \ \forall \ 
\{a_1,\ldots,a_{p-2}\} \subset \{1,\ldots,p\} \}$.

\bibliography{parallel_calib.bib}
\bibliographystyle{plain}

\end{document}